\documentclass[12pt]{article}
\usepackage{amsmath}
\usepackage{amsthm} 
\usepackage{amssymb}
\usepackage{graphicx}
\usepackage{pstricks}
\usepackage{setspace}
\usepackage{maplestd2e}

\DefineParaStyle{Maple Heading 1}
\DefineParaStyle{Maple Text Output}
\DefineParaStyle{Maple Dash Item}
\DefineParaStyle{Maple Bullet Item}
\DefineParaStyle{Maple Normal}
\DefineParaStyle{Maple Heading 4}
\DefineParaStyle{Maple Heading 3}
\DefineParaStyle{Maple Heading 2}
\DefineParaStyle{Maple Warning}
\DefineParaStyle{Maple Title}
\DefineParaStyle{Maple Error}
\DefineCharStyle{Maple Hyperlink}
\DefineCharStyle{Maple 2D Math}
\DefineCharStyle{Maple Maple Input}
\DefineCharStyle{Maple 2D Output}
\DefineCharStyle{Maple 2D Input}
\newtheorem{deff}{Definition}[section]
\newtheorem{twr}[deff]{Theorem}
\newtheorem{lem}[deff]{Lemma}

\newtheorem{coro}[deff]{Corollary}
\newtheorem{example}[deff]{Example}

\newcommand{\E}{\mathbb E}
\newcommand{\N}{\mathbb N}
\newcommand{\I}{\mathbb I}
\newcommand{\diag}{\mathbf{diag}}
\newcommand{\m}[0]{\mathbf{m}}
\newcommand{\e}[0]{\mathbf{e}}
\newcommand{\f}[0]{\mathbf{f}}
\newcommand{\s}[0]{\mathbf{s}}
\newcommand{\PP}[0]{\mathbf{P}}
\newcommand{\C}[0]{\mathbf{C}}
\newcommand{\Q}[0]{\mathbf{Q}}
\newcommand{\R}[0]{\mathbb R}
\newcommand{\X}[0]{\mathbf{X}}
\newcommand{\Hh}[0]{\mathbf{H}}
\newcommand{\dual}[0]{*}

\newcommand{\monD}[0]{\downarrow}
\newcommand{\monU}[0]{\uparrow}
\newcommand{\trev}[0]{\overleftarrow}
\newcommand{\nn}{\textbf n}

\newcommand{\EE}{\mathbb E}

\begin{document}
\title{Strong Stationary Duality for  M\"obius \\ Monotone Markov Chains: Unreliable Networks}

\author{Pawe{\l} Lorek
\thanks{address for both authors: Mathematical
    Institute, University of Wroc{\l}aw,  pl. Grunwaldzki 2/4,  50-384 Wroc{\l}aw, Poland.}
\\{\it  University of Wroc{\l}aw}
\and Ryszard Szekli
\thanks{ Work
supported by MNiSW Research Grant N N201 394137.
}
 \\
  {\it University of Wroc{\l}aw}
} \maketitle

\begin{abstract}
For  Markov chains with a partially ordered finite state space we show strong stationary duality under the condition of M\"obius monotonicity of the chain. We show relations of  M\"obius monotonicity to other definitions of monotone chains. We give examples of dual chains in this context which have transitions only upwards. We illustrate general theory by an analysis of nonsymmetric random walks on the cube with an application to networks of queues.

\end{abstract}

\section{Introduction}
The motivation of this paper stems from a study on the speed of convergence for unreliable queueing networks, as in Lorek and Szekli \cite{lorekszekli}. The problem of bounding the speed of convergence for networks is a rather complex one, and is related to transient analysis of Markov processes, spectral analysis, coupling or duality constructions, drift properties, monotonicity properties, among others (see for more details Dieker and Warren \cite{dieker}, Aldous \cite{aldous88}, Lorek and Szekli \cite{lorekszekli}). In order to give bounds on the speed of convergence for an unreliable queueing network it is necessary to study the availability coordinate of the network process which is a Markov chain with the state space representing  stations with a \emph{down} status in the network. In other words the chain under study is a Markov chain for which the state space  is the power set of the set of nodes, that is a chain representing a random walk on the vertices of the finite dimensional cube. We are especially interested in walks on the cube which are up-down in the natural (inclusion) ordering. To be more precise,
recall that the classical {\bf Jackson
network} consists of $M$ numbered servers, denoted by
$ {J}:=\{1,\ldots,d\}$. Station $j\in {J}$ is a single server queue with
infinite waiting room under FCFS (First Come First Served) regime. All the
customers in the network are indistinguishable. There is an external Poisson
arrival stream with intensity $\lambda$ and arriving customers are sent
to node $j$ with probability $r_{0j}$, $\sum_{j=1}^dr_{0j}=r\leq 1$. Customers arriving at node $j$ from the outside
or from other nodes request a service which is  at node $j$  provided with intensity $\mu_j(n)$
($\mu_j(0):=0$), where $n$ is the number of customers at node $j$ including the
one being served. All the service times and arrival processes are assumed to be
independent. \par A customer departing from node $i$ immediately proceeds to
node $j$ with probability $r_{ij}\geq 0$ or departs from the network with
probability $r_{i0}$. The routing is independent of the past of the system given the
momentary node where the customer is. We
assume that the  matrix $R:=(r_{ij}, \ i,j\in  {J})$ is irreducible.

Let $Y_j(t)$ be the number of customers present at node $j$, at time $t\geq 0$.
Then $$Y(t)=(Y_1(t),\ldots,Y_d(t))$$ is the joint queue length vector at time
instant $t\geq 0$ and ${\bf Y}:=(Y(t),t\geq 0)$ is the joint queue length
process with the state space $\mathbb{Z}_+^d$.

The unique stationary distribution for ${\bf Y}$ exists if and only if the unique solution of
the {\bf traffic equation}
\begin{equation}\label{eq:traffic}
\lambda_i=\lambda r_{0i}+\sum_{j=1}^d \lambda_j r_{ji}, \quad i=1,\ldots,d
\end{equation}
satisfies
$$C_i:= 1+\sum_{n=1}^\infty {{\lambda}_i^n\over \prod_{y=1}^n \mu_i(y)} <\infty, \quad 1\leq i \leq d.$$

Assume that the servers at the nodes in the Jackson network are
unreliable, i.e., the nodes may break down. The breakdown event may occur in
different ways. Nodes may break down as an isolated event or in groups
simultaneously, and the repair of the nodes may end for each node individually
or in groups as well. It is not required that those nodes which stopped service
simultaneously return to service at the same time instant. To describe the
system's evolution we have to enlarge the state space for the network process
as it will be described below. Denote by $ {J}_0:=\{0,1,\ldots,d\}$  the set of nodes enlarged by adding the \emph{outside} node.
\begin{itemize}
 \item  Let $ {D}\subset  {J}$ be the set of servers out of order, i.e. in \emph{down status} and $ {I}\subset {J}\setminus {D},   {I}\neq\emptyset$ be the subset of nodes in \emph{up status}.
Then the servers in $ {I}$ break down with intensity
$\alpha^{ {D}}_{ {D}\cup {I}}( n_i: i\in  {J})$.

\item  Let $ {D}\subset  {J}$ be the set of servers in \emph{down status} and
$ {H}\subset {D},  {H}\neq\emptyset$. The broken servers from
$ {H}$ return from repair with intensity
$\beta_{ {D}\setminus {H}}^{ {D}}(n_i : i\in  {J})$.

\item The routing is changed according to so-called {\sc Repetitive Service -
Random Destination Blocking} (RS-RD BLOCKING) rule: For  $ {D}$  - set of
servers under repair routing probabilities are restricted to nodes from
$ {J}_0\setminus  {D}$ as follows:
$$
r^{ {D}}_{ij}=\left\{
\begin{array}{lll}
 r_{ij}, & i,j\in {J}_0\setminus  {D}, & i\neq j, \\
 r_{ii}+\sum_{k\in  {D}} r_{ik}, & i\in  {J}_0\setminus {D},& i=j.\\
\end{array}\right.
$$
The external arrival rates are
\begin{equation}\label{external_arrivals}
\lambda r^{ {D}}_{0j}=\lambda r_{0j}
\mathrm{\  for\ nodes\  } j\in {J}\setminus {D},
\end{equation}
and zero, otherwise. Let $R^{ {D}}=(r^{ {D}}_{ij})_{i,j\in {J}_0\setminus  {D}}$ be the modified routing.
Note that $R^{\emptyset}=R$.
\end{itemize}

\label{deff:break_repair_int}
 We assume for the intensities of breakdowns and repairs
 $\emptyset\neq  {H}\subset{ {D}}$ and $\emptyset\neq {I}\subset  {J}\setminus {D}$
 that
$$\begin{array}{lcr}
\alpha^{ {D}}_{ {D}\cup {I}}( n_i: i\in  {J}) &:=& {\psi( {D}\cup {I})\over \psi( {D})},\\
  & &\\
\beta^{ {D}}_{ {D}\setminus {H}}( n_i: i\in  {J}) &:=& {\phi( {D})\over \phi( {D}\setminus{H})},
\end{array}
$$
where $\psi$ and $\phi$ are arbitrary positive functions defined for all subsets of the set of nodes. That means that
 breakdown and
repair intensities depend on the sets of servers but are independent of the particular numbers of
customers present in these servers.

In order to  describe unreliable Jackson networks we need to attach to the
state space $\mathbb{Z}_+^d$ of the corresponding standard network process an
additional component which includes information of availability of the system.
 We consider the following state space
$$\tilde\nn =( {D}, n_1,n_2,\ldots,n_d)\in \mathcal{P}( {J})\times \mathbb{Z}_+^d =: \tilde{\EE}=\E\times \mathbb{Z}_+^d,$$
where $\mathcal{P}( {J})$  denotes the powerset of $ {J}$. The first coordinate in $\tilde\nn $ we call the availability coordinate.

The set $ {D}$ is the set of servers in \emph{down status}. At node $i\in {D}$
there are $n_i$ customers waiting for server being repaired. Denote  possible transitions by

\begin{equation}\label{deff:transfText}
\begin{array}{lll}
T_{ij}\tilde\nn & := & ( {D},n_1,\ldots,n_i-1,\ldots,n_j+1,\ldots, n_d), \\
T_{0 j}\tilde\nn & := & ( {D},n_1,\ldots,n_j+1,\ldots, n_d), \\
T_{i0 }\tilde\nn & := & ( {D},n_1,\ldots,n_i-1,\ldots, n_d), \\
T_{ {H}}\tilde\nn & := & ( {D}\setminus {H},n_1,\ldots, n_d), \\
T^{ {I}}\tilde\nn & := & ( {D}\cup {I},n_1,\ldots, n_d). \\
\end{array}
\end{equation}

The Markov process $\tilde{\bf Y}=((X(t),Y(t)),t\geq 0)$ defined on the state space $\tilde \EE$ by the  infinitesimal
generator
\begin{equation}\label{eq:jack_zawodne_gen}
{\small
\begin{array}{llll}
\displaystyle \tilde{\Q}f(\tilde{\nn})=&\displaystyle \sum_{j=1}^d [f(T_{0
j}\tilde{\nn})-f(\tilde{\nn})]  {\lambda} r^{  D}_{0j} & + &
\displaystyle \sum_{i=1}^d\sum_{j=1}^d [f(T_{ij}\tilde{\nn})-f(\tilde{\nn})]
\mu_i(n_i)r^{  D}_{ij} + \\[15pt]
& \displaystyle \sum_{ {I}\subset {J}} [f(T^{ {I}}\tilde {\nn})-f(\tilde{\nn})]{ \psi( {D}\cup {I})\over \psi( {D})}& +&\displaystyle
\sum_{ {H}\subset {J}} [f(T_{ {H}}\tilde{\nn})-f(\tilde{\nn})]
{ \phi( {D})\over \phi( {D}\setminus{H})}+\\
&\displaystyle \sum_{j=1}^d[f(T_{j0 }\tilde{\nn})-f(\tilde{\nn})]\mu_j(n_j)
r^{  D}_{j0} &
\\

\end{array}
}
\end{equation}
is called {\bf unreliable Jackson network} under RS-RD blocking. The coordinate $X(t)$ is called {\bf availability coordinate}. It takes vales in $\E=\mathcal{P}( {J})$.

Similarly to  classical Jackson networks the invariant distribution for this Markov process can be written in a product form, see Sauer and Daduna \cite{SauDad03}. Moreover the question about speed of convergence for the network process $\tilde{\bf Y}$ can be decomposed into a set of simpler questions about speed of convergence for the chain $(X_n)_{n\ge 0}$ imbedded in the availability Markov process $(X(t), t\ge 0)$, and the arising from the product formula birth and death processes, see e.g. Daduna and Szekli \cite{DadSze08}.

In this paper we shall concentrate our attention on the availability chain.
We utilze strong stationary duality which is a probabilistic approach to the problem of speed of convergence to stationarity for Markov chains introduced by Diaconis and Fill \cite{DiaFil90}. This approach involves strong stationary times introduced earlier by Aldous and Diaconis \cite{AldDia86}, \cite{AldDia87} who gave a number of examples showing useful bounds on the total variation distance for convergence to stationarity in cases where other techniques utilizing eigenvalues or coupling were not easily applicable. A \emph{strong stationary time}  for a Markov chain $(X_n)$ is a stopping time $T$ for this chain for which $X_T$ has stationary distribution and is independent of $T$. Diaconis and Fill \cite{DiaFil90} constructed an absorbing \emph{dual} Markov chain with its absorption time equal to the strong stationary time $T$ for $(X_n)$.

In general, there is no recipe how to construct dual chains.
However, few particular cases are tractable. One of them, as Diaconis and Fill \cite{DiaFil90} (Theorem 4.6) prove, is when the state space  is linearly ordered. In this case,
under the assumption of stochastic monotonicity for the time reversed chain,
and under some (mlr) conditions on the
initial distribution it is possible to construct a dual chain which is a  birth and death process
on the same state space but with absorption,  for which the time to absorption allows for a simpler analysis.

In Section 3 we generalize this construction, using  M\"obius monotonicity instead of stochastic
monotonicity, and using a partially ordered space. This construction is of independent interest.
It turns out that in many cases the resulting dual chain is an analog of pure-birth chains, its transitions
are not-downwards.

Utilization of M\"obius monotonicity involves a general problem of inverting
a sum ranging over a partially ordered set which appears in many combinatorial contexts see e.g. Rota \cite{rota}. The inversion can be carried out by defining an analog of the difference operator relative to
a given partial ordering. Such an operator is the M\"obius function, and the analog of the fundamental theorem of calculus obtained
in this context is the M\"obius inversion formula on a partially ordered set, which we recall in Section \ref{mobius}.




\section{M\"obius function and monotonicity}\label{mobius}
Consider a finite, partially ordered set $\E=\{\e_1,\ldots,\e_M\}$, and denote a partial order on $\E$,  by $\preceq$. We select the above enumeration of $\E$  to be consistent with the partial order, i.e. $\e_i\preceq \e_j$  implies $i<j$.

Let $\X=(X_n)_{n\ge 0}\sim(\nu,\PP)$ be a time homogeneous Markov chain with an initial distribution $\nu$ and its transition function $\PP$ on the state space $\E$. We identify the transition function with the corresponding matrix written for the fixed enumeration of the state space.  Suppose that $\X$ is ergodic with the stationary distribution $\pi$.

We shall use $\wedge$ for meet (greatest lower bound) and $\vee$
for the join (least upper bound) in $\E$. If $\E$ is a lattice, it has the unique
minimal and the unique maximal elements, denoted by $\e_1:=\hat{\mathbf{0}}$ and $\e_M:=\mathbf{\hat{1}},$  respectively.


Recall that the \emph{zeta function} $\zeta$ of the partially ordered set $\E$ is defined by:  $\zeta(\e_i,\e_j)=1$ if $\e_i\preceq \e_j$ and $\zeta(\e_i,\e_j)=0$ otherwise. If the  states are enumerated in such a way, that $\e_i\preceq\e_j$ implies $i<j$ (assumed in this paper), then $\zeta$
can be represented by an upper-triangular, 0-1 valued  matrix $\C$, which is invertible.
It is well known that $\zeta$ is an element of \emph{the incidence algebra} (see Rota \cite{rota}, p.344), which is invertible in this algebra, and the inverse
to $\zeta$,
denoted by $\mu$ is called \emph{M\"obius function}. Using the enumeration which
defines $\C$, the corresponding matrix describing $\mu$ is given by the usual matrix inverse $\C^{-1}$.

 For the state space $\E=\{\e_1,\ldots,\e_M\}$ with the partial ordering $\preceq$
 we define the following operators acting on all functions
 $f:\E\to \R$
 \begin{equation}\label{eq:Sdf} 
S_{\downarrow}f(\e_i)=\sum _{\e\in\E}f(\e)\zeta(\e,\e_i)=\sum_{\e:\e\preceq \e_i} f(\e) =: F(\e_i),
\end{equation}
 and
\begin{equation}\label{eq:Suf} 
S_{\uparrow}f(\e_i)=\sum _{\e\in\E}\zeta(\e_i,\e)f(\e)=\sum_{\e:\e\succeq \e_i} f(\e) =: \bar{F}(\e_i).
\end{equation}
In the matrix notation we shall use the corresponding bold letters for functions, and we have $\mathbf{F}=\mathbf{f}\C$,
$\mathbf{\bar{F}}=\mathbf{f}\C^T$, where $\mathbf{f}=(f(\e_1),\ldots,f(\e_M))$,
$\mathbf{F}=(F(\e_1),\ldots,F(\e_M))$, and $\mathbf{\bar F}=(\bar F(\e_1),\ldots,\bar F(\e_M))$.

The following difference operators $D_{\downarrow}$ and $D_{\uparrow}$ are the inverse operators to the summation operators $S_{\downarrow}$ and $S_{\uparrow}$, respectively
\begin{equation}\label{eq:Ddf} 
D_{\downarrow}f(\e_i)=\sum _{\e\in\E}f(\e)\mu(\e,\e_i)=\sum_{\e:\e\preceq \e_i} f(\e)\mu(\e,\e_i)=:g(\e_i),
\end{equation}
 and
\begin{equation}\label{eq:Duf}
D_{\uparrow}f(\e_i)=\sum _{\e\in\E}\mu(\e_i,\e)f(\e)=\sum_{\e:\e\succeq \e_i} \mu(\e_i,\e)f(\e)=:h(\e_i).
\end{equation}
In the matrix notation  we have $\mathbf{g}=\mathbf{f}\C^{-1}$, and
$\mathbf{h}=\mathbf{f}(\C^T)^{-1}$.

If, for example, the  relations
(\ref{eq:Sdf}) and (\ref{eq:Suf}) hold
then
\begin{equation}\label{eq:Mob_inv_preceq}
 f(\e_i)=\sum_{\e:\e\preceq \e_i} F(\e)\mu(\e,\e_i)=D_{\downarrow}(S_{\downarrow}f(\e_i)),
\end{equation}
and
\begin{equation}\label{eq:Mob_inv_succeq}
 f(\e_i)=\sum_{\e:\e\succeq \e_i} \mu(\e_i,\e)\bar F(\e)= D_{\uparrow}(S_{\uparrow}f(\e_i)),
\end{equation}
respectively.

\begin{deff}

For a Markov chain $\X$ with the transition function $\PP$, we say that  $\PP$ (or alternatively that $\X$ ) is
\begin{itemize}
\item[]
$^\monD$-M\"obius monotone if
\begin{equation}\label{mmdown}
\C^{-1}\PP \C \ge 0,
\end{equation}
\item[]
$^\monU$-M\"obius monotone  if
\begin{equation}\label{mmup}
(\C^T)^{-1}\PP \C^T \ge 0,
\end{equation}
\end{itemize}
where $\PP$ is the matrix of the transition probabilities written using the enumeration which defines
$\C$, and $\geq 0$ means that each entry of a matrix is non-negative.
\end{deff}
\begin{deff}
Function $f: \E \to \R$ is
 \begin{itemize}
\item[]
$^\monD$-M\"obius monotone\ if \  $ \mathbf{f}(\C^T)^{-1}\geq 0,$
\item[]
 $^\monU$-M\"obius monotone\  if \ $\mathbf{f}\C^{-1}\geq 0$.
\end{itemize}
\end{deff}

For example, in terms of the M\"obius function $\mu$ and the transition probabilities,
$\monD$- M\"obius monotonicity of $\PP$ means that for all $(\e_i,\e_j \in \E)$
$$\sum_{\e:\e\succeq \e_i} \mu (\e_i,\e) P(\e,\{\e_j\}^\monD)\geq 0,$$
where $P(\cdot,\cdot)$ denotes the corresponding transition kernel, i.e. $P(\e_i,\{\e_j\}^\monD)=\sum_{\e:\e\preceq \e_j} {\PP}(\e_i,\e),$ and
$\{\e_j\}^\monD =\{\e:\e\preceq \e_j\}$. In order to check such a condition an explicit formula
for $\mu$ is needed. Note that the above definition for monotonicity can be rewritten as follows, $f$ is {$^\monD$-M\"obius monotone\ if \ for some non-negative vector $\m \ge 0$, it holds  $\mathbf{f}=\m \C^T,$ and $f$ is
$^\monU$-M\"obius monotone\  if \ $\mathbf{f}=\m\C$. The last equality means that $f$ is a non-negative linear combination of the rows of matrix $\C$. This monotonicity implies that $f$ is non-decreasing in the usual sense ($f$ non-decreasing means: $ \e_i\preceq \e_j$ implies $f(\e_i)\le f(\e_j)$).

Two probability measures   $\pi_1, \pi_2\in\E$
are (strongly) stochastically ordered (we write $\pi_1\preceq_{st}\pi_2$)
when $\pi_1(A)\leq \pi_2(A)$ for all  {\sl upper sets} $A\subseteq\E$
i.e. sets such that $A=\{ \e : \e\succeq \e_i$ for some $\e_i$ in $A\}$.

$\PP$ is \emph{stochastically monotone} if one of the equivalent conditions holds
\begin{itemize}
\item[(i)]  if $\pi_1\preceq_{st}\pi_2$
then $\pi_1\PP\preceq_{st}\pi_2\PP$, where $\pi_1\PP$
denotes the usual multiplication (vector by matrix)
with the fixed enumeration of coordinates,
\item[(ii)]  if $\e_i\preceq \e_j$ then $P(\e_i,A)\le P(\e_j,A)$, for all upper sets $A$,
\item[(iii)]  if $f$ is non-decreasing then $\PP \mathbf{f}^T$ is non-decreasing .
\end{itemize}
For linearly ordered spaces, conditions (\ref{mmup}) and (\ref{mmdown}) are equivalent and
they define the above (strong)
stochastic monotonicity of $\PP$ (see Keilson and Kester \cite{keilson;kester:77}).

For arbitrary partially ordered spaces
strong stochastic monotonicity need not imply M\"obius monotonicity, as will be seen later in this paper.
Also M\"obius monotonicity does not imply the strong stochastic monotonicity in general. However, as noted by
Massey \cite{massey:87}, M\"obius monotonicity implies a weak stochastic monotonicity.
To be more precise, let us recall the definitions of  weak stochastic monotonicities.
\begin{deff}
 \item[]
 \begin{itemize}
 \item[]
  $\PP$ is $^\monD$- weakly monotone if for
 all $\pi_1\preceq_{\downarrow}\pi_2$ it holds $\pi_1\PP\preceq_{\downarrow}\pi_2\PP$, where $\pi_1\preceq_{\downarrow}\pi_2$ when for
all $\e_i\in\E$, $\pi_1(\{\e_i\}^{\downarrow})\le \pi_1(\{\e_i\}^{\downarrow}$)
\item[]
 $\PP$ is $^\monU$- weakly
monotone if for all $\pi_1\preceq_{\monU}\pi_2$ it holds $\pi_1\PP\preceq_{\monU}\pi_2\PP$,
where $\pi_1\preceq_{\monU}\pi_2$ if $\pi_1(\{\e_i\}^{\uparrow})\leq\pi_2(\{\e_i\}^{\uparrow})$
holds for all $\{\e_i\}^{\uparrow}:=
\{\e\in\E : \e_i\preceq \e\}.$
\end{itemize}
\end{deff}
It is possible to characterize $^\monU$- weak
monotonicity of $\PP$ as an invariance property when $\PP$ is acting on the following subset $\cal F$ of functions: ${\cal F}:= \{ f:  \pi_1 \mathbf{f}^T\le \pi_2 \mathbf{f}^T , \forall \ \ \pi_1\preceq_{\monU}\pi_2\}$. We skip an analogous formulation for $^\monD$- weak
monotonicity of $\PP$.
\begin{lem}
$\PP$ is $^\monU$- weakly
monotone iff
\begin{itemize}

\item[]  $f\in\cal F$ implies $\PP \mathbf{f}^T\in\cal F$, for all $f$.
\end{itemize}
\end{lem}
\begin{proof}
Suppose that $f\in\cal F$ implies $\PP \mathbf{f}^T\in\cal F$, and take arbitrary $\pi_1\preceq_{\monU}\pi_2$. For arbitrary $f\in\cal F$ we have then $\pi_1 \mathbf{f}^T\le \pi_2 \mathbf{f}^T$, and $\pi_1 \PP\mathbf{f}^T\le \pi_2 \PP\mathbf{f}^T$. Since $f(\e)=\I_{\{\{\e_i\}\monU\}}(\e)\in {\cal F}$ we have $\pi_1 \PP(\{\e_i\}^{\uparrow})\le \pi_2 \PP(\{\e_i\}^{\uparrow})$, that is $\pi_1\PP\preceq_{\monU}\pi_2\PP$. Conversely, suppose that for all $\pi_1\preceq_{\monU}\pi_2$ implies $\pi_1\PP\preceq_{\monU}\pi_2\PP$. Take arbitrary $f\in {\cal F}$. Then for all $\pi_1\preceq_{\monU}\pi_2$, $\pi_1 \mathbf{f}^T\le \pi_2 \mathbf{f}^T$, and $\pi_1 \PP\mathbf{f}^T\le \pi_2 \PP\mathbf{f}^T$, which implies that $\PP\mathbf{f}^T\in {\cal F}$.
\end{proof}

Similar orderings to the weak stochastic orderings $\preceq_{\monU}$, and  $\preceq_{\monD}$
were studied by Xu and Li \cite{xuli}
for distributions on the d-dimensional cube.
It is reasonable to consider an ordering defined by requiring both weak
stochastic orderings at the same time, which results in a  kind of dependency order
(see Xu and Li \cite{xuli} for details).
\begin{lem}\label{lem:Mw}
\begin{itemize}
\item[]
\item[] $^\monU$-M\"obius monotonicity of $\PP$ implies $^\monU$- weak
monotonicity of $\PP$
\item[] and
\item[] $^\monD$-M\"obius monotonicity of $\PP$ implies $^\monD$- weak
monotonicity of $\PP$.
\end{itemize}
\end{lem}
\begin{proof}
We shall prove only the first implication, the second one can be obtained by replacing $\C^T$ with $\C$ in the argument.
Let $\mathbf{M}=(\C^T)^{-1}\PP\C^T\geq 0$. We have to show
the $^\monU$weak monotonicity of $\PP$ which
is equivalent to the property that  $\pi_1\PP\preceq_{\uparrow}\pi_2\PP,$
for all $\pi_1\preceq_{\monU}\pi_2$.  Suppose that $\pi_1\preceq_{\monU}\pi_2$, this means in terms of $\C$ that $\pi_1\C^T \leq \pi_2\C^T$, coordinatewise. Note that  $^\monU$-M\"obius monotonicity of $\PP$ is equivalent to $\PP\C^T=\C^T \mathbf{M}$. Multiplying the inequality $\pi_1\C^T \leq \pi_2\C^T$ by $\mathbf{M}$ (it is non-negative) we obtain  $\pi_1\C^T\mathbf{M} \leq \pi_2\C^T\mathbf{M}$, and using  $^\monU$-M\"obius monotonicity we get $\pi_1\PP \C^T \leq \pi_2\PP \C^T$, which implies that $\pi_1\PP \preceq_{\monU} \pi_2\PP $.

\end{proof}
The above result for $^\monU$ monotonicity was, independently from Massey \cite{massey:87},  given by Falin \cite{falin}
 in his Theorem 2.

It is also possible to characterize $^\monU$-M\"obius monotonicity
 of $\PP$ as an invariance property when $\PP$ is acting on the set of
all  $^\monU$-M\"obius monotone functions. Note that this set is strictly smaller than $\cal F$. Indeed, taking two probability measures $\pi_1\preceq_{\monU}\pi_2$ on $\{0,1\}^2$, we have always $\pi_2(\{(0,1),(0,0)\})\le \pi_1(\{(0,1),(0,0)\})$, and $\pi_2(\{(1,0),(0,0)\})\le \pi_1(\{(1,0), (0,0)\})$, and for $f$ such that $f((0,1)=-1$, $f((1,0))=-1$, $f((0,0))=-1$ $f((1,1))=0$ we have $ \pi_1 \mathbf{f}^T\le \pi_2 \mathbf{f}^T$, hence $f\in {\cal F}$, but $f$ is not a $^\monU$-M\"obius monotone function. We skip the corresponding formulation for $^\monD$-M\"obius monotonicity
 of $\PP$.
\begin{lem}
$\PP$ is $^\monU$-M\"obius monotone iff
\begin{itemize}
\item[]   $f$ is $^\monU$-M\"obius monotone  implies that $\PP \mathbf{f}^T$ is $^\monU$-M\"obius monotone.
\end{itemize}
\end{lem}
\begin{proof}
Suppose that $\PP$ is $^\monU$-M\"obius monotone, that is $(\C^T)^{-1}\PP\C^T\ge 0$. Take arbitrary $f$ which is $^\monU$-M\"obius monotone, i.e. take $\f=\m\C$ for some arbitrary $\m\ge 0$. Then  $(\C^T)^{-1}\PP\C^T\m^T\ge 0$, which is (using transposition) equivalent to $\f\PP^T\C^{-1}\ge 0$, which in turn gives (from definition) that $\PP\f^T$ is  $^\monU$-M\"obius monotone. Conversely, for all $\f=\m\C$, where $\m\ge0$, we have $\f\PP^T\C^{-1}\ge 0$, since $\PP\f^T$ is  $^\monU$-M\"obius monotone. This implies that
$(\C^T)^{-1}\PP\C^T\m^T\ge 0$, and $(\C^T)^{-1}\PP\C^T\ge 0$.
\end{proof}
In general, weak stochastic monotonicity does not imply M\"obius monotonicity (see Massey \cite{massey:87} or Falin \cite{falin} for counterexamples).

We shall give more examples of M\"obius monotone chains later. However, many examples can be produced using the fact that the set of M\"obius monotone matrices is a convex subset of the set of transition matrices.
\begin{lem}
\begin{itemize}
\item[]
\item[(i)] If $\PP_1$ and $\PP_2$ are $^\monU$-M\"obius monotone ( $^\monD$-M\"obius monotone) then $\PP_1 \PP_2$ is $^\monU$-M\"obius monotone ( $^\monD$-M\"obius monotone).
\item[(ii)] If $\PP$ is $^\monU$-M\"obius monotone ( $^\monD$-M\"obius monotone) then $(\PP)^k$ is $^\monU$-M\"obius monotone ( $^\monD$-M\"obius monotone), for each $k\in \N$.
\item[(iii)] If $\PP_1$ is $^\monU$-M\"obius monotone ( $^\monD$-M\"obius monotone) and $\PP_2$ is $^\monU$-M\"obius monotone ( $^\monD$-M\"obius monotone) then
    $$p\PP_1+(1-p)\PP_2$$
    is $^\monU$-M\"obius monotone ( $^\monD$-M\"obius monotone), for all $p\in (0,1)$.
\end{itemize}
\end{lem}
\begin{proof}
Ad (i). Since $(\C^T)^{-1}\PP_1\C^T\ge 0$, and $(\C^T)^{-1}\PP_2\C^T\ge 0$ then
$$(\C^T)^{-1}\PP_1\PP_2\C^T=((\C^T)^{-1}\PP_1C^T)((C^T)^{-1}\PP_2\C^T)\ge 0.$$
The statements $(ii), (iii)$ are immediate from definition.
\end{proof}
\section{Strong stationary duality for $^\monD$-M\"obius monotone chains}\label{sec:main_result}
\subsection{Time to stationarity}\label{sec:TimeToStat}
One possibility of measuring distance to stationarity is to use
separation distance (see Aldous and Diaconis \cite{AldDia87}). Let $s(\nu\PP^n,\pi)=\max_{\e\in \E}\left(1-
\nu\PP^n(\e)/\pi(\e)\right)$. Separation distance is an upper bound
on total variation distance: $s(\nu\PP^n,\pi)\ge d(\nu\PP^n,\pi):=\max_{B\subset \E}|\nu\PP^n(B)-\pi(B)|$.
\par
A random variable $T$ is a {\em Strong Stationary Time} (SST)  if it is
a randomized stopping time for $\X=(X_n)_{n\ge 0}$ such that $T$ and $X_T$ are independent, and
$X_T$ has distribution $\pi$. SST was introduced by Aldous and Diaconis in \cite{AldDia86,AldDia87}.
In \cite{AldDia87} they prove that $s(\nu\PP^n,\pi)\leq P(T>n)$ ($T$
implicitly depends on $\nu$). Diaconis \cite{Diaconis_book} gives some examples
of bounds on the rates of convergence to stationarity via SST. However, the method how to find
 SST was specific to each example.

Diaconis and Fill \cite{DiaFil90} introduced so-called {\em Strong Stationary Dual} (SSD)
 chains. Such chains have a special feature, namely for them the  SST for the original process has the same distribution as the
time to absorption in the SSD one.

To be more specific,
let $\X^\dual$ be a Markov chain with transition matrix $\PP^\dual$, initial distribution
$\nu^\dual$ on the state space $\E^\dual$. Assume
that $\e^\dual_a$ is an absorbing state for $\X^\dual$.
Let $\Lambda\equiv\Lambda(\e^\dual,\e), \e^\dual\in \E^\dual, \e\in \E$ be a kernel, called a {\em link}, such that  $\Lambda(\e_a^\dual,\cdot)=\pi$ for $\e^\dual_a\in \E^\dual$.
Diaconis and Fill \cite{DiaFil90} prove, that if a $(\nu^*,\PP^*)$ is a SSD of $(\nu,\PP)$ with respect to $\Lambda$ in the sense that
\begin{equation}\label{eq:duality}
\nu=\nu^\dual\Lambda \quad \mbox{ and } \quad \Lambda\PP=\PP^\dual\Lambda,
\end{equation}
then there exists a bivariate Markov chain $(\X,\X^*)$ with the following marginal properties:
\begin{enumerate}
 \item[] $\X$ is  Markov  with the initial distribution $\nu$ and the transition matrix $\PP$,
 \item[] $\X^\dual$ is Markov  with the initial distribution $\nu^\dual$ and the transition matrix $\PP^\dual$,
\item[] the absorption time $T^\dual$ of $\X^\dual$ is a SST for $\X$.
\end{enumerate}

This means, that once one finds a SSD chain for a given chain, then the problem of finding the distribution of SST for the original chain
  can be translated into the problem of finding the distribution of the time to absorbtion in SSD chain. However, there is no general recipe on how to find a link, and a SSD chain. It is also not clear that the absorbtion time in a SSD chain will always be easier to study than the SST time in the original chain. However the case of linearly ordered spaces is a convincing example that it is worth trying to realize such a scenario. In this case, if the original chain is a birth and death process then it is possible to find SSD which is a birth-birth process with absorption. In this paper we find an analog of this situation, starting with some up-down chains we find SSD chains which do not jump downwards or even chains which are jumping only upwards. Details will be presented in the next section.
\subsection{Main result}
Recall that  $\trev{{\bf X}}=(\trev X_n)_{n\ge 0}$ is the time reversed process if its transition matrix is given by
$$
\trev\PP=(\diag(\pi))^{-1}\PP^T(\diag(\pi)),
$$
where $\diag(\pi)$ denotes the matrix which is diagonal with the stationary vector $\pi$ on the diagonal.

\begin{twr}\label{twr:main1}
 Let $\X\sim(\nu,\PP)$ be an ergodic Markov chain on a finite state space $\E=\{\e_1,\ldots,\e_M\}$,
partially ordered by $\preceq$, with an unique maximal state $\e_M$, and
 with the
 stationary distribution $\pi$. Assume that
\begin{itemize}\label{eq:mupi_Mob}
 \item[(i)] $g(\e)={\nu(\e)\over \pi(\e)}$  is $^\monD$-M\"obius monotone,
 \item[(ii)] $\trev\X$ is $\monD$-M\"obius monotone.\label{eq:rev_Mob}
\end{itemize}
Then there exists Strong Stationary Dual chain $\X^\dual\sim(\nu^\dual,\PP^\dual)$ on $\E^\dual=\E$ with the following link kernel
$$\Lambda(\e_j,\e_i)=\I(\e_i\preceq \e_j) {\pi(\e_i)\over H(\e_j)},$$
where $H(\e_j)=S_{\downarrow} \pi(\e_j)=\sum_{\e:\e\preceq \e_j}\pi(\e)$ ($\Hh=\pi\C$). Moreover, the SSD is uniquely determined by
\begin{itemize}
\item[]
 $ \nu^\dual(\e_i)=H(\e_i)\sum_{\e:\e\succeq \e_i}\mu(\e_i,\e)g(\e)=S_{\downarrow} \pi(\e_i)D^\monU g(\e_i) ,\ \ \ \ \  \e_i\in\E,$
 \item[]
 $\PP^\dual(\e_i,\e_j)=
 \frac{H(\e_j)}{ H(\e_i)}\sum_{\e:\e\succeq \e_j} \mu(\e_j,\e)\trev P(\e,\{\e_i\}^\monD)=\frac{S_{\downarrow} \pi(\e_j)}{ S_{\downarrow} \pi(\e_i)}D^\monU \trev P(\e_j,\{\e_i\}^\monD), \ \  \  \e_i,\e_j\in \E.$
\end{itemize}
The corresponding matrix formulas are given by
\begin{align}
\nu^\dual &= \mathbf{g}(\C^T)^{-1}\diag(\pi\C),   \\[5pt] \label{eq:main1_matrix_nu}
\PP^\dual &= (\diag(\Hh)\ \C^{-1}\trev{\PP}\C \ \diag(\Hh)^{-1})^T=
 \\
  &=\diag(\pi\C)^{-1} (\C^T\diag(\pi)) \PP (\C^T\diag(\pi))^{-1} \diag(\pi\C), \label{eq:main1_matrix_trans}
\end{align}
where $\mathbf{g}=(g(\e_1,\ldots, g(e_M))$ (row vector).
\end{twr}
\begin{proof}
We have to check the conditions  (\ref{eq:duality}).
The  first condition given in (\ref{eq:duality}) reads for arbitrary $\e_i\in\E$
\begin{equation}\label{eq:mue}
 \nu(\e_i)  =   \sum_{ \e\succeq \e_i} \nu^\dual(\e){\pi(\e_i)\over H(\e)}
\end{equation}
which is equivalent to
$${\nu(\e_i)\over \pi(\e_i)}= \sum_{\e: \e\succeq \e_i}{ \nu^\dual(\e) \over H(\e)}.$$
From  the M\"obius inversion formula (\ref{eq:Mob_inv_succeq}) we get
$$
{ \nu^\dual(\e_i) \over H(\e_i)}=\sum_{\e: \e\succeq \e_i}\mu(\e_i,\e){\nu(\e)\over \pi(\e)},
$$
which gives the required formula. From the assumption that $g$ is $\monD$-M\"obius monotone it follows that $\nu^\dual\ge 0$. Moreover, since $\nu^\dual=\nu \Lambda$, and $\Lambda$ is a transition kernel, it is clear that $\nu^\dual$ is a probability vector.

The second condition given in (\ref{eq:duality}) means that for all $\e_i,\e_j\in \E$
$$\sum_{\e\in\E} \Lambda(\e_i,\e)\PP(\e,\e_j)=\sum_{\e\in\E} \PP^\dual(\e_i,\e)\Lambda(\e,\e_j).$$
Taking the proposed $\Lambda$ we have to check that
$$\sum_{\e:\e\preceq \e_i} {\pi(\e)\over H(\e_i)} \PP(\e,\e_j)
=\sum_{\e:\e\succeq \e_j} {\pi(\e_j)\over H(\e)}\PP^\dual(\e_i,\e), $$
that is
$${1\over H(\e_i)}\sum_{\e:\e\preceq \e_i} {\pi(\e)\over \pi(\e_j)} \PP(\e,\e_j)
=\sum_{\e:\e\succeq \e_j} {\PP^\dual(\e_i,\e)\over H(\e)}. $$
Using ${\pi(\e)\over\pi(\e_j)}\PP(\e,\e_j)=\trev{\PP}(\e_j,\e)$ we have
\begin{equation}\label{eq:rev_sum01}
{1\over H(\e_i)}  \trev P(\e_j, \{\e_i\}^\monD)
=\sum_{\e:\e\succeq \e_j} {\PP^\dual(\e_i,\e)\over H(\e)}.
\end{equation}
For each fixed $\e_i$ we treat ${1\over H(\e_i)}  \trev P(\e_j, \{\e_i\}^\monD)$ as a function of $\e_j$ and again use the  M\"obius inversion formula (\ref{eq:Mob_inv_succeq}) to get from (\ref{eq:rev_sum01})
$${\PP^\dual(\e_i,\e_j)\over H(\e_j)}
=\sum_{\e:\e\succeq \e_j} \mu(\e_j,\e) {\trev P(\e, \{\e_i\}^\monD)\over H(\e_i)}. $$
In the matrix notation we have
$$
\C^{-1}\PP\C (\e_i,\e_j)=\sum_{\e:\e\succeq\e_i}\mu(\e_i,\e)P(\e,\{\e_j\}^\monD),
$$
therefore
$$
\PP^\dual = (\diag(\Hh)\ \C^{-1}\trev{\PP}\C \ \diag(\Hh)^{-1})^T.
$$
Since, from our assumption, $\C^{-1}\trev{\PP}\C\ge 0$, we have $\PP^\dual\ge 0$. Now $\Lambda \PP=\PP^\dual \Lambda$ implies that $\PP^\dual$ is a transition matrix.
\end{proof}

In a similar way we construct an analog SSD chain for $^\monU$-M\"obius monotone $\PP$. We skip the corresponding matrix formulation and a proof.
\begin{coro}\label{twr:main2}
Let $\X\sim(\nu,\PP)$ be an ergodic Markov chain on a finite state space $\E=\{\e_1,\ldots,\e_M\}$,
partially ordered by $\preceq$, with an unique minimal state $\e_1$, and
 with the
 stationary distribution $\pi$. Assume that
\begin{itemize}\label{eq:mupi_Mob}
 \item[(i)] $g(\e)={\nu(\e)\over \pi(\e)}$  is $^\monU$-M\"obius monotone,
 \item[(ii)] $\trev\X$ is $\monU$-M\"obius monotone.
\end{itemize}
Then there exists Strong Stationary Dual chain $\X^\bullet\sim(\nu^\bullet,\PP^\bullet)$ on $\E^\bullet=\E$ with the following link
$$\Lambda^\bullet(\e_j,\e_i)=\I(\e_i\succeq \e_j) {\pi(\e_i)\over \bar H(\e_j)},$$
where $\bar H(\e_j)=S_{\monU} \pi(\e_j)$. Moreover, the SSD is uniquely determined by
\begin{itemize}
\item[]
 $ \nu^\bullet(\e_i)=\bar H(\e_i)\sum_{\e:\e\preceq \e_i}g(\e)\mu(\e,\e_i)=S_{\uparrow} \pi(\e_i)D^\monD g(\e_i) ,\ \ \ \ \  \e_i\in\E,$
 \item[]
 $\PP^\bullet(\e_i,\e_j)=
 \frac{\bar H(\e_j)}{\bar H(\e_i)}\sum_{\e:\e\preceq \e_j} \trev P(\e,\{\e_i\}^\monU)\mu(\e,\e_j)=\frac{S_{\uparrow} \pi(\e_j)}{ S_{\uparrow} \pi(\e_i)}D^\monD \trev P(\e_j,\{\e_i\}^\monU), \ \  \  \e_i,\e_j\in \E.$
\end{itemize}
\end{coro}
For $\E=\{1,\ldots,M\}$ with the linear ordering $\le$ we obtain the Theorem 4.6 from Diaconis and Fill \cite{DiaFil90} as a special case.  We use the fact that in this case $\mu(k,k)=1, \mu(k-1,k)=-1$ and $\mu$ equals $0$ otherwise.
\begin{coro}
 Let $\X\sim(\nu,\PP)$ be an ergodic Markov chain on a finite state space $\E=\{1,\ldots,M\}$,
linearly ordered by $\le$,
 with the
 stationary distribution $\pi$. Assume that
\begin{itemize}\label{eq:mupi_Mob}
 \item[(i)] $g(i)={\nu(i)\over \pi(i)}$  is non-increasing,
 \item[(ii)] $\trev\X$ is stochastically monotone.
\end{itemize}
Then there exists Strong Stationary Dual chain $\X^\dual\sim(\nu^\dual,\PP^\dual)$ on $\E^\dual=\E$ with the following link kernel
$$\Lambda(j,i)=\I(i\le j) {\pi(i)\over H(j)},$$
where $H(j)=\sum_{k:k\le j}\pi(k)$ . Moreover, the SSD is uniquely determined by
\begin{itemize}
\item[]
 $ \nu^\dual(i)=H(i)\sum_{k:k\ge i}\mu(i,k)g(k)=H(i)(g(i)-g(i+1)) ,\ \ \ \ \  i\in\E,$
 \item[]
 $$\PP^\dual(i,j)=
 \frac{H(j)}{ H(i)}\sum_{k:k\ge j} \mu(j,k)\trev P(k,\{1,\ldots ,i\})=$$$$\frac{H(j)}{ H(i)}(\trev P(j,\{1,\ldots,i\})-\trev P(j+1,\{1,\ldots,i\})), \ \  \  i,j\in \E.$$
\end{itemize}
\end{coro}
An analog of the above result, corresponding to $\monU$-M\"obius monotonicity is as follows.
\begin{coro}
 Let $\X\sim(\nu,\PP)$ be an ergodic Markov chain on a finite state space $\E=\{1,\ldots,M\}$,
linearly ordered by $\le$,
 with the
 stationary distribution $\pi$. Assume that
\begin{itemize}\label{eq:mupi_Mob}
 \item[(i)] $g(i)={\nu(i)\over \pi(i)}$  is non-decreasing,
 \item[(ii)] $\trev\X$ is stochastically monotone.
\end{itemize}
Then there exists Strong Stationary Dual chain $\X^\bullet\sim(\nu^\bullet,\PP^\bullet)$ on $\E^\bullet=\E$ with the following link kernel
$$\Lambda^\bullet(j,i)=\I(i\ge j) {\pi(i)\over \bar H(j)},$$
where $\bar H(j)=\sum_{k:k\ge j}\pi(k)$ . Moreover, the SSD is uniquely determined by
\begin{itemize}
\item[]
 $ \nu^\bullet(i)=\bar H(i)\sum_{k:k\le i}\mu(k,i)g(k)=\bar H(i)(g(i)-g(i-1)) ,\ \ \ \ \  i\in\E,$
 \item[]
 $$\PP^\bullet(i,j)=
 \frac{\bar H(j)}{\bar H(i)}\sum_{k:k\le j} \mu(k,j)\trev P(k,\{i,\ldots ,M\})=$$$$\frac{\bar H(j)}{ \bar H(i)}(\trev P(j,\{i,\ldots,M\})-\trev P(j-1,\{i,\ldots,M\})), \ \  \  i,j\in \E.$$
\end{itemize}
\end{coro}
{\bf Remarks.}

1. In Theorem \ref{twr:main1} we have $\Lambda(\e_M,\cdot)=\pi$ and  $\e_M$ is
an absorbing state for $\PP^\dual$. Moreover, if the orignal chain
starts with probability 1 in the minimal state, i.e. $\nu=\delta_{\e_1}$, so does the
dual chain. In Corollary \ref{twr:main2} we have $\Lambda^\bullet(\e_1,\cdot)=\pi$ and  $\e_1$ is
an absorbing state for the dual $\PP^\bullet$. Moreover, if the orignal chain
starts with probability 1 in the maximal state, i.e. $\nu=\delta_{\e_M}$, so does the
dual chain.

2. A well-known theorem, usually attributed to Keilson, states that, for an irreducible continuous-time birth-and-death chain on $\E=\{0,\ldots,M\}$, the passage time from state 0 to state M is distributed as a sum of M independent exponential random variables. Fill \cite{Filla} uses the theory of strong stationary duality to give a stochastic proof of an analogous result for discrete time birth and death chains and geometric random variables. He shows a link for the parameters of the distributions to eigenvalue information about the chain. The obtained dual is a pure birth chain.

3. An (upward) skip-free Markov chain with the set of nonnegative integers as state space is a chain for which upward jumps may be only of unit size; there is no restriction on downward jumps. In Brown and Shao \cite{BroShao} determined, for an irreducible continuous-time skip-free chain and any $M$, the passage time distribution from state $0$ to state $M$. When the  eigenvalues  of the generator  are all real, their result states that the passage time is distributed as the sum of $M$ independent exponential random variables with rates equal to the eigenvalues . Fill \cite{Fillb}  gives another proof of this theorem. In the case of birth-and-death chains, this proof leads to an explicit representation of the passage time as a sum of independent exponential random variables. Diaconis and Miclo \cite{DiaMiclo} recently obtained  such a representation, using an involved duality construction.


\subsection{Nearest neighbor M\"obius monotone walks on cube: one station repair or failure}
Consider discrete time Markov chain ${\bf X}=\{X_n, n\geq 0\}$, $\E=\{0,1\}^d$ with the transition matrix $\PP$ given by,

\begin{equation}\label{rw_cube}
\begin{array}{rclcl}
 \PP(\e,\e+\s_i) & = & \alpha_i \I_{\{e_i=0\}}, \\
\\
 \PP(\e,\e-\s_i) & = & \beta_i \I_{\{  e_i=1\}}, \\
\\
 \PP(\e,\e) & = &\displaystyle 1-\sum_{i:e_i=0}\alpha_i- \sum_{i:e_i=1}\beta_i,   \\
\end{array}
\end{equation}
where $\e=(e_1,\ldots,e_d)\in \E, \ \  e_i\in\{0,1\}$ and  $\s_i=(0,\ldots,0,1,0,\ldots,0)$ with $1$ at the position $i$.

Assume that $\alpha_i$ and $\beta_i$ are such that the chain is ergodic.

This  chain is time-reversible with the stationary distribution
\begin{equation}\label{eq:cube_3d_stat_dist}
\pi(x)=\prod_{i:x_i=1}{\alpha_i\over \alpha_i+\beta_i} \prod_{i:x_i=0}{\beta_i\over \alpha_i+\beta_i}.
\end{equation}

We use the following partial ordering: $\e=(e_1,e_2,\ldots,e_d) \preceq \e' = (e'_1,e'_2,\ldots,e'_d)$ if $e_i\leq e'_i$, for all $i=1,\ldots d$.   Let $|\e|=\sum_{i=1}^d e_i$. For this ordering M\"obius function is known and is given by
$$
\mu(\e,\e')=\left\{
\begin{array}{lll}
(-1)^{|\e'|-|\e|} & \mbox{if } \e\preceq \e', \\[7pt]
0 & \mbox{otherwise.}
\end{array}
\right.
$$
We shall calculate

$$\PP^\dual(\e_i,\e_j)=
 \frac{H(\e_j)}{ H(\e_i)}\sum_{\e:\e\succeq \e_j} \mu(\e_j,\e)\trev P(\e,\{\e_i\}^\monD)$$

and find conditions for its non-negativity.
$$\PP^*((0,\ldots,0),(0,\ldots,0))=\sum_{\e\succeq (0,\ldots,0)} \mu((0,\ldots,0),\e)\trev P(\e,\{(0,\ldots,0)\}^\monD)$$

 $$=1-(\alpha_1+\ldots+\alpha_d)-\beta_1-\ldots-\beta_d=1-\sum_{i=1}^d\alpha_i-\sum_{i=1}^d\beta_i$$
Thus, we must have $$\sum_{i=1}^d\alpha_i+\sum_{i=1}^d\beta_i\leq 1.$$
Note that this condition is equivalent to the condition that all eigenvalues of $\PP$ are non-negative.

Fix $\s_{i}=(0,\ldots,0,1,0,\ldots,0)$ with $1$ on the position $i$. Then
$$\PP^*(\s_{i},\s_{i})=\sum_{\e:\e\succeq \s_{i}} \mu(\s_{i},\e) \trev P(\e,\{\s_i\}^\monD)=1-\sum_{k=1}^d\alpha_k+\alpha_i-\sum_{k=1}^d{\beta_k}+\beta_i.$$

For each state of the form  $\e^i=(e_1,\ldots,e_{i-1},0,e_{i+1},\ldots,e_d),$
$$\PP^*(\e^i,\e^i+\s_i)={H(\e^i+\s_i)\over H(\e^i)}\sum_{\e:\e\succeq \e^i+\s_i} \mu(\e^i+\s_i,\e)\trev P(\e ,\{ \e^i\}^\monD)
$$
$$={H(\e^i+\s_i)\over H(\e^i)}\left( \mu(\e^i+\s_i,\e^i+\s_i)\trev P(\e^i+\s_i ,\{ \e^i\}^\monD)\right)={H(\e^i+\s_i)\over H(\e^i)}\beta_i$$

Denote by $z(\e)=\{k: e_k=0\}$, the index set of zero coordinates.
We shall compute ${H(\e^i+\s_i)\over H(\e^i)}$. Let $G=\prod_{j=1}^d(\alpha_j+\beta_j)$.

$$H(\e^i)=\sum_{\e:\e\preceq \e^i} \pi(\e)={1\over G} \sum_{\e:\e\preceq \e^i} \prod_{k: e_k=1} \alpha_k \prod_{k: e_k=0} \beta_k={1\over G} \prod_{k: e^i_k=0} \beta_k \left(\sum_{A\subseteq \{1,\ldots,d\}\setminus z(\e^i)} \prod_{j\in A} \alpha_j \prod_{j\in A^C} \beta_j\right)$$
$$H(\e^i+\s_i)=\sum_{\e\preceq \e^i+\s_i} \pi(\e)=H(\e^i)+\sum_{\e\preceq\e^i+\s_i\atop e_i=1} \pi(\e)$$
$$=H(\e^i)+{1\over G} \prod_{k: e_k=0} \beta_k \alpha_i \left(\sum_{A\subseteq \{1,\ldots,d\}\setminus z(\e^i)} \prod_{j\in A} \alpha_j \prod_{j\in A^C} \beta_j\right)$$
$$=H(\e^i)+{1\over G} \prod_{k: e^i_k=0} \beta_k {\alpha_i\over \beta_i} \left(\sum_{A\subseteq \{1,\ldots,d\}\setminus z(\e^i)} \prod_{j\in A} \alpha_j \prod_{j\in A^C} \beta_j\right)$$
And thus
$${H(\e^i+\s_i)\over H(\e^i)}=1+{\alpha_i\over \beta_i}={\alpha_i+\beta_i\over \beta_i},$$
and
$$
\PP^*(\e^i,\e^i+\s_i)=\alpha_i+\beta_i.
$$

Now, fix some $^i\e=\{e_1,\ldots,e_{i-1},1,e_{i+1},\ldots,e_d\}$.
$$\PP^*(^i\e, ^i\e-\s_i)={H(^i\e-\s_i)\over H(^i\e)}\sum_{\e:\e\succeq \ ^i\e-\s_i} \mu(^i\e-\s_i,\e)\trev P(\e , \{^i\e\}^\monD)
$$
Fix $j\in\{1,\ldots,d\}\setminus z(^i\e)$. The following cases are possible:
$$
\begin{array}{lllllll}
 \e=\ ^i\e-\s_i :&\quad & \mu(^i\e-\s_i,\ ^i\e-\s_i)\trev P(^i\e-\s_i,\{^i\e\}^\monD) &=\displaystyle 1-\sum_{k\in z(^i\e)} \alpha_k \\
 \e=\ ^i\e :&\quad & \mu(^i\e-\s_i,\ ^i\e)\trev P(^i\e,\{^i\e\}^\monD) &=\displaystyle - \left(1-\sum_{k\in z(^i\e)} \alpha_k\right)\\[6pt]
 \e=\ ^i\e-\s_i+\s_j :&\quad & \mu(^i\e-\s_i,\ ^i\e-\s_i+\s_j)\trev P(\ ^i\e-\s_i+\s_j,\{^i\e\}^\monD) &=\displaystyle - \beta_j \\[7pt]
 \e=\ ^i\e+\s_j :&\quad & \mu(^i\e-\s_i,^i\e+\s_j)\trev P(\ ^i\e+\s_j,\{^i\e\}^\monD) &=\displaystyle \beta_j \\
\end{array}
$$

Summing up all possibilities we get
$$\PP^*(^i\e, ^i\e-\s_i)=0.$$

For each $\e$ we have
$$\PP^\dual(\e,\e)=\sum_{\e':\e'\succeq \e} \mu(\e,\e')\trev P(\e',\{\e\}^\monD)$$
$$
\begin{array}{lllllll}
 \e'=\e :&\quad & \mu(\e,\e)\trev P(\e,\{\e\}^\monD) &=\displaystyle 1-\sum_{i\in z(\e)} \alpha_i \\
 \e'=\e+\s_i :&\quad & \mu(\e,\e+\s_i)\trev P(\e+\s_i,\{\e\}^\monD) &=\displaystyle -1\cdot \beta_i.\\
\end{array}
$$

Therefore we get
$$\PP^\dual(\e,\e)=\mu(\e,\e)\trev P(\e,\{\e\}^\monD)+\sum_{i\in z(\e)}
\mu(\e,\e+s_i)\trev P(\e+\s_i,\{\e\}^\monD)=1-\sum_{i\in z(\e)}(\alpha_i+\beta_i).$$

It is interesting that the dual (absorbing)  chain here is a chain which jumps only upwards to neighboring states or stay at the same state. This structure of the dual chain allows to read all eigenvalues for $\PP$ and $\PP^\dual$ (from the diagonal of $\PP^\dual$ ) since the dual matrix  $\PP^\dual$ is an upper-triangular matrix. The symmetric walk $\alpha_i=\beta_i=(1-r)/d,$ was considered by Diaconis and Fill \cite{DiaFil90}. They used the symmetry to reduce the problem to a birth and death chain. The non-symmetric case was studied by Brown \cite{Brown90a} were the eigenvalues were identified by a different method. Using Brown \cite{Brown90a} we can reformulate his result as follows.

\begin{twr}\label{twr:cube_st_monot_cond1} Suppose ${\bf X}=\{X_n,n\geq 0 \}$ is  defined by (\ref{rw_cube}). Assume that $\sum_{i=1}^d\alpha_i+\sum_{i=1}^d\beta_i\leq 1.$ Define
$A_k$ to be a set of  ${d \choose k}$ subsets of size $k$ from $\{1,\ldots,d\}$ and $s_\gamma=\sum_{i\in \gamma} (\alpha_i +\beta_i)$ for $\gamma$ a subset of ${1,\ldots,d}$.
Then for the separation distance, and $\delta_{(0,\ldots,0)}$ the atomic measure at $(0,\ldots,0)$ we have
\begin{eqnarray}
 (i) &\hspace{20pt}\displaystyle s(\delta_{(0,\ldots,0)}\PP^n,\pi)=\sum_{k=1}^d (-1)^{k-1} \sum_{\gamma\in A_k} (1-s_\gamma)^n, \qquad n\geq  1,\label{twr:cons_cube_wyl_sep}\\
(ii) &\hspace{20pt} {\mathrm all\ } 2^d {\mathrm \ eigenvalues \ of \ } \PP {\mathrm \ are \ } \{1-s_\gamma,\gamma\subseteq\{1,\ldots,d\}\}.\label{twr:cons_cube_egen}
\end{eqnarray}
\end{twr}

We recognize $(i)$ as an inclusion-exclusion formula:
consider $n$ multinomial trials with cell probabilities $p_i=\alpha_i+\beta_i$, $i=1,\ldots,d,$
and $p_{d+1}=1-\sum_{i=1}^d (\alpha_i+\beta_i)$. Let $A_n$ be the event that at least one of the cells from $1,\ldots,d$ is empty.
Then $A_n=\cup_{i=1}^d C_i$, where $C_i=1$ if cell $i$ is empty, 0 otherwise.
Then (i) represents the inclusion-exclusion formula for $P(\cup_{i=1}^d C_i)$.
Thus if we denote $T$ to be the waiting time for all of cells $1,\ldots,d$ to be occupied, then
$$s(\delta_{x_{min}} \PP^n,\pi)=P(T>n)=P(A_n).$$

\medskip
\begin{example}\rm
 Consider the following random walk on the two-dimensional
cube: $\E=\{0,1\}^2$. We define on $\E$ the   partial ordering: for all $\e=(e_1,e_2)\in \E$, $\e'=(e'_1,e'_2)\in \E$

 $\e\preceq \e'\iff e_1\leq e_1', e_2\leq e_2'$.

We consider the transition matrix under the state space enumeration:

$$\e_1=(0,0), \e_2=(1,0), \e_3=(0,1), \e_4=(1,1),$$

of the form
$$\PP=\left[
\begin{array}{cccc}
 1-\alpha_1-\alpha_2 & \alpha_1 & \alpha_2 & 0 \\
\beta_1 & 1-\beta_1-\alpha_2 & 0 & \alpha_2 \\
\beta_2 &  0   & 1-\alpha_1-\beta_2  & \alpha_1 \\
0 & \beta_2 & \beta_1 & 1-\beta_1-\beta_2\\
\end{array}\right].
$$
We assume that $\alpha_1,\alpha_2,\beta_1,\beta_2$ are  positive and
that there exists $ B\subseteq \E$ such that $\sum_{k\in B} \alpha_k
+\sum_{k\in B^c} \beta_k>0$ (which assures irreducibility of the chain).

The zeta, and the M\"obius functions are represented by

$$\C=\left[
\begin{array}{rrrr}
 1 & 1 & 1 & 1  \\
 0 & 1 & 0 & 1\\
 0 & 0 & 1 & 1\\
 0 & 0 & 0 & 1\\
\end{array}\right]
\quad
\C^{-1}=\left[
\begin{array}{rrrr}
 1 & -1 & -1 & 1 \\
 0 & 1 & 0 & -1\\
 0 & 0 & 1 & -1\\
0 & 0 & 0 & 1\\
\end{array}\right].
$$
\begin{itemize}
 \item[] For $^\monU$-M\"obius monotonicity we have
$$
(\C^T)^{-1} \PP \C^T =
\left[
\begin{array}{cccc}
1 & \alpha_1 & \alpha_2 & 0\\
0 & 1-\alpha_1-\beta_1 & 0 & \alpha_2 \\
0 & 0 & 1-\alpha_2-\beta_2 & \alpha_1\\
0 & 0 & 0 &  1-\alpha_1-\alpha_2-\beta_1-\beta_2 \\
\end{array}\right].
$$
$\PP$ is $^\monU$-M\"obius monotone iff
\begin{equation}\label{eq:cube2d_mob_mon}
\sum_i (\alpha_i+\beta_i)\leq 1.
\end{equation}

 \item[] For $^\monD$-M\"obius monotonicity we have
$$\C^{-1}\PP \C=
\left[
\begin{array}{cccc}
 1-\alpha_1-\alpha_2-\beta_1-\beta_2 & 0 & 0 & 0 \\
\beta_1 & 1-\alpha_2-\beta_2 & 0 & 0\\
\beta_2 &0& 1-\alpha_1-\beta_1 & 0\\
0 & \beta_2 &\beta_1 & 1\\
\end{array}\right].
$$
 Again, $\PP$ is $^\monD$-M\"obius monotone iff $\sum_i (\alpha_i+\beta_i)\leq 1$. Notice that the above matrix is similar to $\PP$ and have an upper-triangular form, so the eigenvalues of $\PP$ are given on the diagonal of this matrix. If we assume that $\alpha_1=\alpha_2=\alpha$, and $\beta_1=\beta_2=\beta$ then the dual matrix $\PP^\dual$ has also a simple form (again with the eigenvalues on the diagonal),

\begin{maplelatex}
\mapleinline{inert}{2d}{Matrix(
\end{maplelatex}

We have moreover

\begin{maplelatex}
\mapleinline{inert}{2d}{Matrix(
\end{maplelatex}
\end{itemize}
\end{example}

\subsection{More general walks on cube}

Consider first the  random walk on the three-dimensional
cube: $\E=\{0,1\}^3$. We define on $\E$ the   partial ordering: for all $\e=(e_1,e_2,e_3)\in \E$, $\e'=(e'_1,e'_2,e'_3)\in \E$, $\e\preceq \e'$ iff $e_1\leq e_1', e_2\leq e_2', e_3\leq e'_3$.

We consider the transition matrix under the state space enumeration:

$\e_1=(0,0,0), \e_2=(1,0,0), \e_3=(0,1,0), \e_4=(0,0,1), \e_5=(1,1,0),\e_6=(1,0,1),\e_7=(0,1,1),\e_8=(1,1,1)$
of the form

\begin{maplelatex}\tiny
\mapleinline{inert}{2d}{Matrix(
\end{maplelatex}

\bigskip
with the dual $\PP_1^\dual$

\bigskip
\begin{maplelatex}\tiny
\mapleinline{inert}{2d}{Matrix(
\end{maplelatex}

One possibility to extend the model to allow up-down jumps not only to neighboring states, that is to model repairs and failures of more than one station at one transition of the process, is to take powers of the nearest neighbor transitions matrix $\PP_1$. The matrix $\PP=\PP_1^2$ is again M\"obius monotone, and has the dual with an upper-triangular form. To be able to present this matrix in display form we take $\alpha=\beta$.

\medskip
\begin{maplelatex}\tiny
\mapleinline{inert}{2d}{Matrix(
\end{maplelatex}

\bigskip
Another way to modify the nearest neighbor walk is to transform some rows of $\PP$ to get distributions bigger in the supermodular ordering. To be more precise, recall that
we say that two random elements $X,Y$ of $\E$ are supermodular
stochastically ordered (and write $X\prec_{sm}Y$ or $Y\succ _{sm}
X$) if $ E f(X) \leq E f(Y)$ for all supermodular
functions, i.e. functions which
fulfill for all $x,y\in \E$
$$f(x\wedge y) +
f(x\vee y) \geq f(x)+ f(y).$$

\bigskip
A simple sufficient criterion for $\prec_{sm}$ order for $\E$
which is a discrete (countable) lattice is given as follows.
\begin{lem}\label{p2.8.a}
Let  $P_1$ be  a probability measure on  a discrete lattice
ordered space $\E$ and assume that for not comparable points
$x\neq y\in \E$ we have $P_1(x) \geq \kappa$ and $P_1(y)\geq
\kappa$ for some $\kappa
> 0$. Define a new probability measure $P_2$ on  $\E$ by
\begin{eqnarray}\label{2.7.a}
P_2(x) =  {} ~P_1(x) - \kappa &\qquad\qquad&
P_2(x\vee y) =  {}~ P_1(x\vee y) +  \kappa \nonumber\\
P_2(y) =  {}~ P_1(y) - \kappa &\qquad\qquad&
P_2(x\wedge y) =  {} ~P_1(x\wedge y) + \kappa \nonumber\\
P_2(z) =  {} ~P_1(z) \qquad&\text{otherwise}.&
\end{eqnarray}
Then $P_1\prec_{sm} P_2$.\\
\end{lem}

If in Lemma \ref{p2.8.a} the state space $\E$ is the set of all
subsets of a finite set (i.e. the cube) than
the transformation described in \eqref{2.7.a} is called in Li and
Xu \cite{xuli} a {\em  pairwise} $g^+$ {\em transform} and
Lemma \ref{p2.8.a} specializes then to Proposition 5.5., Li and Xu
\cite{xuli}.

If we modify rows numbered 1,3,6,8 by such a transformation (notice that $\e_1,\e_3,\e_6,\e_8$ lie on a symmetry axis), that is we consider an up-down walk which allows jumps not only to the nearest neighbors with the transition matrix

\bigskip
\begin{maplelatex}\tiny
\mapleinline{inert}{2d}{Matrix(
\end{maplelatex}

\bigskip
then the dual matrix again has an upper-triangular form

\bigskip
\begin{maplelatex}\tiny
\mapleinline{inert}{2d}{Matrix(
\end{maplelatex}

\bigskip
which allows to read the eigenvalues ( for appropriate selection of $\alpha$ and $\kappa$).

There are several other examples of chains which are M\"obius monotone, with transitions to non-comparable states on the cube, and walks on some other state spaces, but we shall study this topic  in a subsequent paper.

\end{document}